\documentclass[10pt,leqno]{amsart}
\usepackage{indentfirst, amssymb,amsmath,amsthm}
\theoremstyle{definition}
\newtheorem*{theoA}{Theorem A}

\newtheorem{theo}{Theorem}[section]
\newtheorem{lem}{Lemma}[section]

\newtheorem{defi}{Definition}[section]

\newcommand{\ol}{\overline}
\newcommand{\be}{\begin{equation}}
\newcommand{\ee}{\end{equation}}
\newcommand{\beas}{\begin{eqnarray*}}
\newcommand{\eeas}{\end{eqnarray*}}
\newcommand{\bea}{\begin{eqnarray}}
\newcommand{\eea}{\end{eqnarray}}
\newcommand{\lra}{\longrightarrow}
\numberwithin{equation}{section}
\begin{document}
\title[ A simple proof of the Chuang's inequality]{A simple proof of the Chuang's inequality}
\date{}
\author[B. Chakraborty ]{Bikash Chakraborty }
\date{}
\address{ Department of Mathematics, University of Kalyani, West Bengal,  India-741235.}
\address{Department of Mathematics, Ramakrishna Mission Vivekananda Centenary College, Rahara, India-700 118}
\email{bikashchakraborty.math@yahoo.com, bikashchakrabortyy@gmail.com.}
\maketitle
\let\thefootnote\relax
\footnotetext{Type set by \AmS -\LaTeX}
\setcounter{footnote}{0}
\begin{abstract} The purpose of the paper is to present an short proof of the Chuang's inequality (\cite{3a}).
\end{abstract}
\section{Introduction Definitions and Results}
Originally, the Chuang's inequality is a standard estimate in Nevanlinna theory but later on this inequality is used as a valuable tool in the study
of value distribution of differential polynomials. For example, Recently, using this inequality, some sufficient conditions are obtained for which two differential polynomials sharing a small function satisfies the conclusions of Br\"{u}ck conjecture (\cite{bb}, \cite{bc1}, \cite{bc2}, \cite{bm}). \par
At this point, we recall some notations and definitions to proceed further.\par
It will be convenient to let $E$ denote any set of positive real numbers of finite linear measure, not necessarily the same at each occurrence.\par
Let $f$ be a non constant meromorphic function in the open complex plane $\mathbb{C}$. For any non-constant meromorphic function $f$, we denote by $S(r, f)$ any quantity satisfying $$S(r, f) = o(T(r, f))\;\;\;\;\;\;\;\;\;\;\ (r\lra \infty, r\not\in E).$$
A meromorphic function $a(\not\equiv \infty)$ is called a small function with respect to $f$ provided that $T(r,a)=S(r,f)$ as $r\lra \infty, r\not\in E$.\par
We use $I$ to denote any set of infinite linear measure of $0<r<\infty$.\\
Now  we recall the following definition.
\begin{defi} (\cite{4})Let $n_{0j},n_{1j},\ldots,n_{kj}$ be non negative integers.\\
The expression $M_{j}[f]=(f)^{n_{0j}}(f^{(1)})^{n_{1j}}\ldots(f^{(k)})^{n_{kj}}$ is called a differential monomial generated by $f$ of degree $d(M_{j})=\sum\limits_{i=0}^{k}n_{ij}$ and weight
$\Gamma_{M_{j}}=\sum\limits_{i=0}^{k}(i+1)n_{ij}$.

The sum $P[f]=\sum\limits_{j=1}^{t}b_{j}M_{j}[f]$ is called a differential polynomial generated by $f$ of degree $\ol{d}(P)=max\{d(M_{j}):1\leq j\leq t\}$
and weight $\Gamma_{P}=max\{\Gamma_{M_{j}}:1\leq j\leq t\}$, where $T(r,b_{j})=S(r,f)$ for $j=1,2,\ldots,t$.

The numbers $\underline{d}(P)=min\{d(M_{j}):1\leq j\leq t\}$ and k (the highest order of the derivative of $f$ in $P[f]$ are called respectively the lower degree and order of $P[f]$.

$P[f]$ is said to be homogeneous if $\ol{d}(P)$=$\underline{d}(P)$.

$P[f]$ is called a Linear Differential Polynomial generated by $f$ if $\ol {d} (P)=1$. Otherwise $P[f]$ is called Non-linear Differential Polynomial. We also denote by $\mu =max\; \{\Gamma _{M_{j}}-d(M_{j}): 1\leq j\leq t\}=max\; \{ n_{1j}+2n_{2j}+\ldots+kn_{kj}: 1\leq j\leq t\}$.
\end{defi}
Now we are position to state the Chuang's inequality.
\begin{theoA} \label{l2.4} (\cite{3a}) Let $f$ be a meromorphic function and $P[f]$ be a differential polynomial. Then
$$ m\left(r,\frac{P[f]}{f^{\ol {d}(P)}}\right)\leq (\ol {d}(P)-\underline {d}(P)) m\left(r,\frac{1}{f}\right)+S(r,f).$$
\end{theoA}
In this article, we give a short proof of the above inequality with some restriction.
\begin{theo} \label{l2.5} Let $f$ be a meromorphic function and $P[f]$ be a differential polynomial. Then
$$ m\left(r,\frac{P[f]}{f^{\ol {d}(P)}}\right)\leq (\ol {d}(P)-\underline {d}(P)) m\left(r,\frac{1}{f}\right)+S(r,f)~\text{as}~ r \to \infty~\text{and}~r \not\in E_{0},$$
where $E_{0}$ is a set whose linear measure is not greater than $2$.
\end{theo}
 For this, we need to recall the lemma of logarithmic derivative.
\begin{lem}[\emph{Lemma of Logarithmic Derivative}] (\cite{1234})
Suppose that $f(z)$ is a non constant meromorphic function in whole complex plane. Then
$$m\bigg(r,\frac{f'}{f}\bigg)=S(r,f)~\text{as}~ r \to \infty~\text{and}~r \not\in E_{0},$$
where $E_{0}$ is a set whose linear measure is not greater than $2$.
\end{lem}
\begin{lem}\label{ghgh}
Suppose that $f(z)$ is a non constant meromorphic function in whole complex plane and $l$ is natural number. Then
$$m\bigg(r,\frac{f^{(l)}}{f}\bigg)=S(r,f)~\text{as}~ r \to \infty~\text{and}~r \not\in E_{0},$$
where $E_{0}$ is a set whose linear measure is not greater than $2$.
\end{lem}
\section {Proof of Chuang's inequality}
\begin{proof} [Proof of theorem \ref{l2.5}]
Suppose that $P[f]=\sum_{j=1}^{t}b_{j}M_{j}[f]$ be a differential polynomial generated by a non constant meromorphic function $f$. Further suppose that $m_{j}=d(M_{j})$ for $j=1,2,...,t$. Without loss of any generality, we can assume that $m_1\leq m_2\leq...\leq m_t$.\par We have to prove this inequality by induction on $t$.\\
If $t=1$, then in view of Lemma \ref{ghgh} the inequality follows. Next we assume that the inequality holds for $t=l(\geq2)$. Now we have to show that the inequality holds for $t=l+1$.\\
Assume $$P[f]=\sum\limits_{j=1}^{l+1}b_{j}M_{j}[f]=Q[f]+bM[f],$$
where $Q[f]=\sum\limits_{j=1}^{l}b_{j}M_{j}[f]$, $M[f]=M_{l+1}[f]$ and $b=b_{l+1}$.\\
Then $m_1\leq m_2\leq...\leq m_l\leq m_{l+1}$ and by hypothesis $$m\left(r,\frac{Q[f]}{f^{\ol {d}(Q)}}\right)\leq (\ol {d}(Q)-\underline {d}(Q)) m\left(r,\frac{1}{f}\right)+S(r,f)~\text{as}~ r \to \infty~\text{and}~r \not\in E_{0},$$
where $E_{0}$ is a set whose linear measure is not greater than $2$.\par
Thus \beas m\left(r,\frac{P[f]}{f^{\ol {d}(P)}}\right)&=&m\left(r,\frac{Q[f]+b M[f]}{f^{\ol {d}(P)}}\right)\\
&\leq& m\left(r,\frac{Q[f]}{f^{\ol {d}(P)}}\right)+m\left(r,\frac{ M[f]}{f^{\ol {d}(P)}}\right)+S(r,f)\\
&\leq& (\ol {d}(Q)-\underline {d}(Q)) m\left(r,\frac{1}{f}\right)+(\ol {d}(P)-\ol {d}(Q)) m\left(r,\frac{1}{f}\right)\\
&+&(\ol {d}(P)-d(M)) m\left(r,\frac{1}{f}\right)+S(r,f)\\
&\leq& (\ol {d}(P)-\underline {d}(P)) m\left(r,\frac{1}{f}\right)\\
&+& (\ol {d}(P)+\underline {d}(P)-\ol{d}(Q)-d(M)) m\left(r,\frac{1}{f}\right)+S(r,f)\\
&\leq& (\ol {d}(P)-\underline {d}(P)) m\left(r,\frac{1}{f}\right)+S(r,f)
\eeas
as $r \to \infty~\text{and}~r \not\in E_{0}$, where $E_{0}$ is a set whose linear measure is not greater than $2$, and
$(\ol {d}(P)+\underline {d}(P)-\ol{d}(Q)-d(M))=m_{l+1}+m_{1}-m_{l}-m_{l+1}\leq 0.$\\
Thus by the principle of Mathematical Induction, the inequality follows.
\end{proof}


\begin{thebibliography}{99}
\bibitem{bb} A. Banerjee and M. B. Ahamed, Meromorphic function sharing a small function with its differential polynomial. Acta Univ. Palack. Olomuc. Fac. Rerum Natur. Math. 54(1) (2015), 33-45.
\bibitem{bc1} A. Banerjee and B. Chakraborty, On the generalizations of Br\"{u}ck conjecture, \emph{Commun. Korean Math. Soc.}, 31(2) (2016), 311-327.
\bibitem{bc2} A. Banerjee and B. Chakraborty, Some further study on Br\"{u}ck conjecture, \emph{An. S\c{t}iin\c{t}. Univ. Al. I. Cuza Ia\c{s}i Mat. (N.S.)}, 62(2.2) (2016), 501-511.
\bibitem{bm} A. Banerjee and S. Mallick, Br\"{u}ck conjecture -- a different perspective. Commun. Fac. Sci. Univ. Ank. S\'{e}r. A1 Math. Stat. 65(1) (2016), 71-86.
\bibitem{bh} S. S. Bhoosnurmath; M. N. Kulkarni and K. W. Yu, On the value distribution of differential polynomials, Bull. Korean Math. Soc., 45(3) (2008), 427--435.
\bibitem{br3} R. Br\"{u}ck, On entire functions which share one value CM with their first derivative, \emph{Results Math.}, 30 (1996), 21-24.
\bibitem{3a} C. T. Chuang, On differential polynomials, Analysis of one complex variable (Laramie, Wyo., 1985) 12-32, World Sci. Publishing Singapore 1987.
\bibitem{4} W. K. Hayman, Meromorphic Functions, The Clarendon Press, Oxford (1964).
\bibitem{1234} C. C. Yang and H. X. Yi, Uniqueness theory of meromorphic functions, \emph{Kluwer Academic Publishers}, (2003).
\end{thebibliography}
\end{document}